\documentclass[12pt]{article}
\usepackage{booktabs}
\usepackage{caption}
\usepackage{mathrsfs}
\usepackage{amsmath}
\usepackage{amsfonts,amsthm,amssymb,mathrsfs,bbding}
\usepackage{float}
\usepackage{graphics,epstopdf}
\usepackage{graphics,multicol}
\usepackage{graphicx}
\usepackage{color}
\usepackage{enumerate}
\usepackage{caption}
\usepackage{rotating}
\usepackage{lscape}
\usepackage{longtable}
\allowdisplaybreaks[4]
\usepackage{tabularx}
\usepackage[colorlinks=true,anchorcolor=blue,filecolor=blue,linkcolor=blue,urlcolor=blue,citecolor=blue]{hyperref}
\usepackage{extarrows}
\usepackage{cite}
\usepackage{latexsym,bm}
\usepackage{mathtools}

\usepackage[usenames,dvipsnames]{pstricks}
 \usepackage{pstricks-add}
 \usepackage{epsfig}
 \usepackage{pst-grad} 
 \usepackage{pst-plot} 
 \usepackage[space]{grffile} 
 \usepackage{etoolbox} 

\evensidemargin=\oddsidemargin\topmargin=-1.5cm

\newtheorem{thm}{Theorem}[section]

\newtheorem{lem}{Lemma}[section]

\newtheorem{fact}{Fact}[section]

\newtheorem{claim}{Claim}[section]
\newtheorem{definition}{Definition}[section]

\addtocounter{section}{0}
\begin{document}
\title{Outerplanar Tur\'{a}n numbers of cycles and paths\footnote{Supported by the National Natural Science Foundation of China
(No. 12171066), Anhui Provincial Natural Science Foundation (No. 2108085MA13) and NSF of Education Ministry of Anhui Province (No. KJ2020B06).}}
\author{{\bf Longfei Fang}, {\bf Mingqing Zhai}\thanks{Corresponding author: mqzhai@chzu.edu.cn
(M.Zhai)}
\\
{\footnotesize School of Mathematics and Finance, Chuzhou
University, Anhui, Chuzhou, 239012, China}}
\date{}

\date{}
\maketitle
{\flushleft\large\bf Abstract}
A graph is outerplanar
if it can be embedded in a plane such that all vertices lie on its outer face.
The outerplanar Tur\'{a}n number of a given graph $H$,
denoted by ${\rm ex}_{\mathcal{OP}}(n,H)$,
is the maximum number of edges over all outerplanar graphs on $n$ vertices
which do not contain a copy of $H$.
In this paper, the outerplanar Tur\'{a}n numbers of cycles and paths are completely determined.

\begin{flushleft}
\textbf{Keywords:} Tur\'{a}n number; Outerplanar graph; Cycle; Path
\end{flushleft}
\textbf{AMS Classification:}  05C35, 05C38

\section{Introduction}
~~~~All graphs considered in this paper are undirected and simple.
Given a graph $G$, we use $V(G)$ to denote
the vertex set, $E(G)$ the edge set, $|G|$ the number of vertices, $e(G)$ the number of edges, respectively.
For $v\in V(G) $, we denote by $ N_G(v) $ the \textit{neighborhood} of $v$ and $d_G(v)$ the \textit{degree} of $v$.
For $S\subseteq V(G)$,
let $G[S]$ (resp. $G-S$) be the subgraph of $G$ induced by $S$ (resp. $V(G)-S$).
The \textit{complete graph}, \textit{path} and \textit{cycle} on $k$ vertices are denoted by $K_k,P_k$ and $C_k$, respectively.

Given a graph $H$, a graph is said to be \emph{$H$-free} if it contains no $H$ as a subgraph.
As one of the earliest results in extremal graph theory, the Tur\'{a}n Theorem gives
the maximum number of edges in a $K_k$-free graph on $n$ vertices.
The \emph{Tur\'{a}n number} of a graph $H$, denoted by ${\rm ex}(n,H)$, is the maximum number of edges in an $H$-free graph on $n$ vertices.
In 1959, Erd\H{o}s and Gallai \cite{DX} showed that
${\rm ex}(n,P_k)\le \frac{(k-2)n}{2}$ with equality if and only if $k-1~|~n$.
For all $n$ and $k$, ${\rm ex}(n,P_k)$ was determined by Faudree and Schelp \cite{Faudree}
and independently by Kopylov \cite{KOP};
and ${\rm ex}(n,C_{2k+1})$ was determined by F\"{u}redi and Gunderson \cite{Furedi}.
However, the exact value of ${\rm ex}(n,C_{2k})$ is still open.
For more results on Tur\'{a}n-type problem, we refer the readers to a survey \cite{KC}.

In 2015, Dowden \cite{DZ} initiated the study of Tur\'{a}n-type problem when the host graph is planar, i.e., how many edges can an $H$-free planar graph on $n$ vertices have?
The \emph{planar Tur\'{a}n number} of a graph $H$,
denoted by ${\rm ex}_{\mathcal{P}}(n,H)$, is the maximum number of edges in an $H$-free planar graph on $n$ vertices.
As pointed in \cite{DZ},
${\rm ex}_{\mathcal{P}}(n,K_3)=2n-4$ and ${\rm ex}_{\mathcal{P}}(n,K_k)=3n-6$ for every $n\ge 6$ and $k\ge 4$.
Since the planar Tur\'{a}n problem on $K_k$ is trivial for $k\ge 5$,
the next natural type of graphs considered are cycles and others variations.
Dowden \cite{DZ} showed that ${\rm ex}_{\mathcal{P}}(n,C_4)\le \frac{15(n-2)}{7}$ for all $n\ge 4$ with equality when $n\equiv30 \pmod {70}$,
and ${\rm ex}_{\mathcal{P}}(n,C_5)\le \frac{12n-33}{5}$ for all $n\ge 11$ where the bound is also sharp.
By using triangular-block skill,
Ghosh et al. \cite{DG} proved that ${\rm ex}_{\mathcal{P}}(n,C_6)\le \frac{5n}{2}-7$ for $n\ge 18$, with equality when $n\equiv10 \pmod {18}$.
Lan, Shi and Song \cite{D, BT} obtained several sufficient conditions on $H$ which
yield ${\rm ex}_{\mathcal{P}}(n,H)=3n-6$, and studied the planar Tur\'{a}n numbers of wheels, stars,
$(t,r)$-fans and theta graphs.
In \cite{HH}, Lan and Shi determined the planar Tur\'{a}n number of $P_k$ for $6\le k\le 11$,
and in \cite{B20} Lan, Shi and Song gave a conjecture on ${\rm ex}_{\mathcal{P}}(n,P_k)$.
Up to now, both ${\rm ex}_{\mathcal{P}}(n,C_k)$ and ${\rm ex}_{\mathcal{P}}(n,P_k)$ are still open for general $k$.
For more results on planar Tur\'{a}n-type problem, we can see a survey
due to Lan, Shi and Song \cite{B20}.

A graph $G$ is \emph{outerplanar} if it has an embedding in a plane such that all vertices belong to the boundary of its outer face.
In the past decades, a quantity of works have been carried out in the area of outerplanar graphs
(see, for example, \cite{b13,b12,X,Y,b14}).
The \emph{outerplanar Tur\'{a}n number} of a graph $H$, denoted by ${\rm ex}_{\mathcal{OP}}(n,H)$, is the maximum number of edges in an $H$-free outerplanar graph on $n$ vertices.
The \emph{generalized outerplanar Tur\'{a}n number} of $H$, denoted by ${f}_{\mathcal{OP}}(n,H)$,
is the maximum number of copies of $H$ in an $n$-vertex outerplanar graph.
Very recently, Matolcsi and Nagy \cite{Matolcsi} gave an exact value of ${f}_{\mathcal{OP}}(n,P_3)$
and a best asymptotic value of ${f}_{\mathcal{OP}}(n,P_4)$;
Gy\H{o}ri, Paulos and Xiao \cite{Gyori} gave an exact value of ${f}_{\mathcal{OP}}(n,P_4)$
and a best asymptotic value of ${f}_{\mathcal{OP}}(n,P_5)$.

In this paper, we focus on the outerplanar Tur\'{a}n numbers of cycles and paths.
In Sections 2 and 3, we obtain exact values of ${\rm ex}_{\mathcal{OP}}(n,C_k)$ and 
${\rm ex}_{\mathcal{OP}}(n,P_k)$, respectively.

\section{Outerplanar Tur\'{a}n number of cycles}
~~~~
We shall first introduce some notations and definitions on outerplanar graphs.
An \emph{outerplane graph} $G$ is a planar embedding of an outerplanar graph
such that all vertices lie on the boundary of $\tau(G)$, where
$\tau(G)$ is the outer face of $G$.
As usual, $F$ denotes a face of $G$,
and $F$ sometimes also denotes the subgraph induced by edges on its boundary.
An edge $e$ is said to be \emph{incident} with a face $F$ if $e\in E(F)$.
The $\emph{size}$ of a face $F$ is the number of edges in $ E(F)$,
cut edges being counted twice.
A face of size $k$ is called a \emph{$k$-face} and a face of size at most $k-1$ is called a \emph{$k^-$-face}.
There are some known facts on outerplane graphs.

\begin{fact}\label{fact2.1} Let $G$ be an outerplane graph with $|G|\ge 3$.\\
(\romannumeral1) Every inner face (if exists) is a cycle.\\
(\romannumeral2) If $G$ is 2-connected, then $\tau(G)$ is a Hamilton cycle.\\
(\romannumeral3) If $G$ is edge-maximal, then $G$ is a 2-connected graph with $e(G)=2|G|-3$.
\end{fact}

Now we give a definition that we will use in subsequent proof.

\begin{definition}\label{de2.1}
Let $k\ge 3$ and $G$ be an outerplane graph with $e_0,e_l\in E(G)$.
We say that $e_0$ is \emph{$k$-face-connected} to $e_l$ and denote $e_0\sim_k e_l$,
if $e_0=e_l$, or there is a sequence $e_0F_1e_1F_2 \dots F_{l}e_l$  whose terms are alternately edges and  inner $k^-$-faces,
such that both $e_{i-1}$ and $e_i$ are incident with $F_i$ for
each $i\in \{1,\cdots,l\}$.
\end{definition}

\begin{figure}[!ht]
	\centering
	\includegraphics[width=0.4\textwidth]{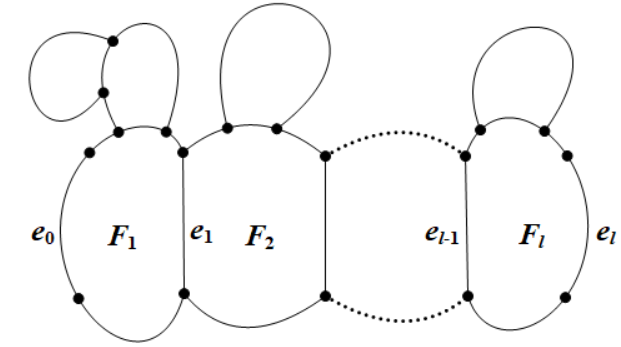}
	\caption{The subgraph $G(k,e)$, where $e=e_0$. }{\label{figu.2}}
\end{figure}

We can see that $k$-face-connectedness is an equivalence relation on $E(G)$.
Denote by $G(k,e)$ the subgraph of $G$ induced by edges in an equivalence class $\{e'\in E(G): e'\sim_k e\}$
(see Fig. \ref{figu.2}).
We have some facts on $G(k,e)$.

\begin{fact}\label{fact2.2}
Let $k\ge 3$ and $G$ be an outerplane graph with $e\in E(G)$. Then\\
(i) every inner $k^-$-face $F$ of $G$ is an inner face of $G(k,e)$ provided that $e\in E(F)$;\\
(ii) every inner face of $G(k,e)$ is an inner $k^-$-face of $G$;\\
(iii) either $G(k,e)\cong K_2$ or $G(k,e)$ is a 2-connected outerplane subgraph;\\
(iv) $G =G(k,e)$ if $G$ is 2-connected and each inner face is a $k^-$-face.
\end{fact}

\begin{proof}
(i) Assume that $e\in  E(F)$.
Since $F$ is an inner $k^-$-face of $G$, the edge $e$ is $k$-face-connected to each edge in $ E(F)$.
It follows that $ E(F)\subseteq E(G(k,e))$.
Note that $G(k,e)$ is an outerplane subgraph of $G$ and $F$ does not contain the infinite-point.
Therefore, $F$ is an inner face of $G(k,e)$.

\vspace{1mm}
(ii) Let $F_0$ be an arbitrary inner face of $G(k,e)$.
By Fact \ref{fact2.1}, $F_0$ is a cycle,
and hence we can select two distinct edges $e_0,e_l\in  E(F_0)$.
Since $e_0,e_l\in  E(F_0)\subseteq E(G(k,e))$,
there exists a sequence $\pi:=e_0F_1e_1F_2 \dots F_{l}e_l$ whose terms are alternately edges and inner $k^-$-faces of $G$,
such that $e_{i-1},e_i\in  E(F_i)$ for each $i\in \{1,\cdots,l\}$.
We may assume that the length of $\pi$ is minimal.
Then, faces in $\pi$ are distinct.

\begin{figure}[!ht]
	\centering
	\includegraphics[width=0.3\textwidth]{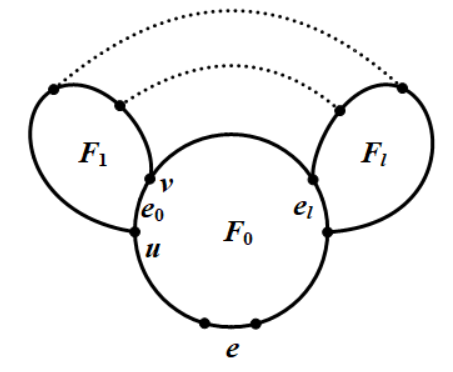}
	\caption{Local structure of $G(k,e)$. }{\label{figu.1}}
\end{figure}

To prove (ii), it suffices to show $F_0\in \{F_1,\dots,F_l\}$.
If not, then we can see that $v$, an endpoint of $e_0$,
does not lie on $\tau(G(k,e))$ (see Fig. \ref{figu.1}),
which contradicts the fact that $G(k,e)$ is outerplane.
Thus, $F_0\in \{F_1,\dots,F_l\}$, as desired.

\vspace{1mm}
(iii) If $G(k,e)\ncong K_2$, then by Definition \ref{de2.1} we can see that $G(k,e)$
has an ear decomposition. By the famous ear decomposition theorem due to Whitney (see \cite{HW}),
$G(k,e)$ is 2-connected.

\vspace{1mm}
(iv) Since $G$ is 2-connected,
we have $|G|\ge 3$, and $e$ is incident with at least one inner face, say $F$, of $G$.
Since every inner face of $G$ is a $k^-$-face, by (i) $ E(F)\subseteq E(G(k,e))$.
By Fact \ref{fact2.1}, $F$ is a cycle, and so $|G(k,e)|\ge3$.
Since $G$ is 2-connected, by Fact \ref{fact2.1} $\tau(G)$ is a Hamilton cycle;
and by (iii) we similarly have that $\tau(G(k,e))$ is a Hamilton cycle.
If $\tau(G)=\tau(G(k,e))$, then by (ii) we have $G=G(k,e)$.
Next, it suffices to show $\tau(G)=\tau(G(k,e))$.

Suppose to the contrary, then there is an edge $e'\in  E(\tau(G(k,e)))\setminus  E(\tau(G))$,
as both $\tau(G)$ and $\tau(G(k,e))$ are cycles.
Since $e'\not\in  E(\tau(G))$, the edge $e'$ is incident with two inner faces $F'$ and $F''$ of $G$.
Note that both $F'$ and $F''$ are $k^-$-faces with $e'\in  E(F')\cap  E(F'')$.
By (i), both $F'$ and $F''$ are inner faces of $G(k,e')$, where $G(k,e')=G(k,e)$.
This contradicts the fact that $e'\in  E(\tau(G(k,e)))$.
\end{proof}

Given an outerplane graph $G$ with $e\in E(G)$.
By Fact \ref{fact2.2},
either $G(k,e)\cong K_2$ or $G(k,e)$ is a 2-connected outerplane graph with each inner face a $k^-$-face.
We call $G(k,e)$ a \emph{$k$-block} of $G$.
Particularly,
if $G\cong K_2$ or $G$ itself is a 2-connected outerplane graph with each inner face a $k^-$-face,
then $G$ is called a \emph{trivial $k$-block}.
By Fact \ref{fact2.2}, we also have $G=G(k,e)$ for any $e\in E(G)$ provided that $G$ is a trivial $k$-block.
In addition, we see that $\{E(G(k,e)): e\in E(G)\}$ is a partition of $E(G)$,
and hence $G$ has a decomposition of its $k$-blocks.
We now give several lemmas.

\begin{lem}\label{lemma2.1}
Let $t,k$ be two integers with $t\ge 1$ and $k\ge 3$.
Let $B_1,\dots,B_t$ be $C_k$-free trivial $k$-blocks.
Then
\begin{align*}
  \min\{|G|:~G\in \mathscr{B}_k(B_1,\dots,B_t)\}\ge \sum_{i=1}^{t}|B_i|-\Big\lfloor\frac{t-1}{k}\Big\rfloor-t+1,
\end{align*}
where $\mathscr{B}_k(B_1,\dots,B_t)$ is the set of connected $C_k$-free outerplane graphs
with a decomposition of $t$ $k$-blocks,
which are isomorphic to $B_1,\dots,B_t$, respectively.
\end{lem}

\begin{figure}[!ht]
	\centering
	\includegraphics[width=0.21\textwidth]{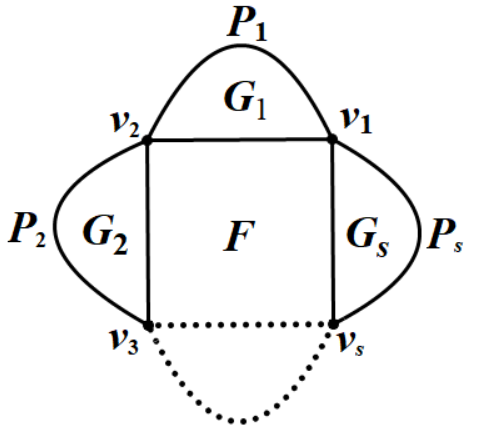}
	\caption{An example of 2-connected outerplane graph $G$. }{\label{figu.3}}
\end{figure}

\begin{proof}
We use the induction on $t$.
Let $G\in \mathscr{B}_k(B_1,\dots,B_t)$.
If $t=1$, then $|G|=|B_1|$ and the inequality holds.
Assume then that $t\ge 2$. It follows that $|G|\geq3.$

If $G$ has cut-vertices,
then there exists an end-block $G_1$,
which contains exactly one cut-vertex $v$ of $G$.
Let $G_2:=G-V(G_1-v)$.
It follows that
\begin{align}\label{align.1}
  |G|=|G_1|+|G_2|-1.
\end{align}
Since each $B_i$ is either isomorphic to $K_2$ or 2-connected,
it must be contained in one of $G_1$ and $G_2$.
We may assume without loss of generality that
$G_1$ has a decomposition of $t_1$  $k$-blocks, which are isomorphic to $B_{1},\dots,B_{t_1}$, respectively.
Then $G_2$ has a decomposition of $t-t_1$ $k$-blocks, which are isomorphic to $B_{t_1+1},\ldots,B_t$, respectively.
Note that $\max\{t_1,t-t_1\} \le t-1$.
By the induction hypothesis, we have
\begin{align}\label{align.2}
 |G_1|\ge \sum_{i=1}^{t_1}|B_i|-\Big\lfloor\frac{t_1-1}{k}\Big\rfloor-t_1+1,
\end{align}
\begin{align}\label{align.3}
 |G_2|\ge\sum_{i={t_1+1}}^{t}|B_i|-\Big\lfloor\frac{t-t_1-1}{k}\Big\rfloor-(t-t_1)+1.
\end{align}
Note that $\lfloor\frac{t_1-1}{k}\rfloor+\lfloor\frac{t-t_1-1}{k}\rfloor \le \lfloor\frac{t-1}{k}\rfloor$.
Combining with (\ref{align.1}-\ref{align.3}), we have
\begin{align*}
  |G|\ge \sum_{i=1}^{t}|B_i|-\Big\lfloor\frac{t-1}{k}\Big\rfloor-t+1.
\end{align*}

It remains to consider that $G$ is 2-connected.
Now if every inner face is a $k^-$-face,
then $G$ is a trivial $k$-block and $G=G(k,e)$ for arbitrary $e\in E(G)$,
which contradicts the fact that $G$ has a decomposition of $t\ge 2$ $k$-blocks.
Thus, $G$ has an inner face $F$ of size $s\ge k$.
Furthermore, by Fact \ref{fact2.1} $F$ is a cycle, and so $s\ge k+1$ as
$G$ is $C_k$-free.
Put $F:=v_{1}v_{2}\dots v_{s}v_{1}$.
By Fact \ref{fact2.1}, $\tau(G)$ is a Hamilton cycle.
Hence, $v_1,\ldots,v_s$ divide the boundary of $\tau(G)$ into
$s$ paths $P_1,\dots,P_s$ (see Fig. \ref{figu.3}).
Let $G_j:=G[V(P_j)]$ for $j\in \{1,\dots,s\}$.
Then $G_1,\dots,G_s$ form a decomposition of $G$.
This implies that
\begin{align}\label{align.4}
 |G|=\sum_{j=1}^{s}|G_j|-s.
\end{align}

Let $e$ be an arbitrary edge of $G$.
Then $e\in E(G_{j_0})$ for some $j_0$.
We need to show $G(k,e)=G_{j_0}(k,e)$.
It suffices to prove that $E(G(k,e))=E(G_{j_0}(k,e))$.
Let $e'$ be an arbitrary edge in $E(G_{j_0}(k,e))$.
If $e'=e$, then $e'\in E(G(k,e))$.
If $e'\neq e$, then there is a sequence $\pi:=eF_1e_1F_2\dots F_{l}e'$
whose terms are alternately edges and inner $k^-$-faces of $G_{j_0}$.
Observe that each inner face of $G_{j_0}$ is an inner face of $G$,
then all faces in $\pi$ are $k^-$-faces of $G$, and so $e'\in E(G(k,e))$.
Therefore, $E(G_{j_0}(k,e))\subseteq E(G(k,e))$.
On the other hand, let $e'\in E(G(k,e))$. If $e'=e$, then $e'\in E(G_{j_0}(k,e))$.
If $e'\neq e$, then there is a sequence $\pi':=eF_1'e_1'F_2'\dots F_{l}'e'$
whose terms are alternately edges and inner $k^-$-faces of $G$.
Since $F$ is of size $s\ge k+1$, we have $F\not\in\{F_1',\dots,F_l'\}$.
Observe that for any two inner faces of $G$, if they belong to two distinct $G_j$s
then they can only be connected by $F$.
Thus, all faces in $\pi'$ are inner $k^-$-faces of $G_{j_0}$, and so $e'\in E(G_{j_0}(k,e))$.
Hence, $E(G(k,e))\subseteq E(G_{j_0}(k,e))$.

We now have that every $k$-block $G(k,e)$ of $G$ is a $k$-block of some $G_j$.
Let $\mathcal{B}_k(G_j)$ be the set of $k$-blocks of $G_j$ for $j\in \{1,\dots,s\}$.
Observe that $E(G_i)\cap E(G_j)=\varnothing$ for any $i,j$ with $i\neq j$.
Hence, $\{\mathcal{B}_k(G_1),\dots,\mathcal{B}_k(G_s)\}$ is a partition of $\mathcal{B}_k(G)$.
Recall that $\mathcal{B}_k(G)=\{B_1,\dots,B_t\}$.
Then
\begin{align}\label{align.5}
 \sum_{j=1}^{s}\sum_{B\in \mathcal{B}_k(G_j)}|B|=\sum_{i=1}^{t}|B_i|~~~\mbox{and}~~~t=\sum_{j=1}^st_j,
\end{align}
where $t_j=|\mathcal{B}_k(G_j)|$.
By the induction hypothesis, we have
\begin{align}\label{align.6}
 |G_j|\ge \sum_{B\in \mathcal{B}_k(G_j)}|B|-\Big\lfloor\frac{t_j-1}{k}\Big\rfloor-t_j+1
\end{align}
for $j\in \{1,\dots,s\}$.
From $t=\sum_{j=1}^st_j$ and $s\ge k+1$, we have
\begin{align}\label{align.7}
\sum_{j=1}^{s}\Big\lfloor\frac{t_j-1}{k}\Big\rfloor \le \Big\lfloor\sum_{j=1}^{s}\frac{t_j-1}{k}\Big\rfloor
=\Big\lfloor \frac{t-s}{k}\Big\rfloor
\le\Big\lfloor\frac{t-1}{k}\Big\rfloor-1.
\end{align}
Combining with (\ref{align.4}-\ref{align.7}), we see that
\begin{align*}
 |G|\ge \sum_{j=1}^{s}\sum_{B\in \mathcal{B}_k(G_j)}|B|-\sum_{j=1}^{s}\Big\lfloor\frac{t_j-1}{k}\Big\rfloor-t
 \ge \sum_{i=1}^{t}|B_i|-\Big\lfloor\frac{t-1}{k}\Big\rfloor-t+1,
\end{align*}
as desired.
\end{proof}


\begin{lem}\label{lemma2.2}
Let $t,k$ be two integers with $t\ge 1$ and $k\ge 3$.
Let $B_1,\dots,B_t$ be $C_k$-free trivial $k$-blocks.
Then
\begin{align*}
  \min\{|G|:~G\in \mathscr{B}_k(B_1,\dots,B_t)\}= \sum_{i=1}^{t}|B_i|-\Big\lfloor\frac{t-1}{k}\Big\rfloor-t+1.
\end{align*}
\end{lem}

\begin{proof}
By Lemma \ref{lemma2.1},
it suffices to find a graph $G\in\mathscr{B}_k(B_1,\dots,B_t)$ with $|G|=\sum_{i=1}^{t}|B_i|-\lfloor\frac{t-1}{k}\rfloor-t+1$.
More precisely, $G$ is a connected $C_k$-free outerplane graph
with a decomposition of $k$-blocks $B_1,\dots,B_t$.

Let $a=\lfloor\frac{t-1}{k}\rfloor$ and $b=t-ka$.
Then $ka\le t-1$ and  $b\ge 1$.
Let $G_0$ be a graph obtained from the disjoint union of $B_1,\dots,B_b$ and
a path $v_{0,1}\dots v_{0,b+1}$ by identifying an edge of
$\tau(B_{i})$ and $v_{0,i}v_{0,i+1}$ for each $i\in \{1,\dots,b\}$.
Then $G_0$ is $C_k$-free and
\begin{align}\label{align.8}
|G_0|=\sum_{i=1}^{b}|B_i|-(b-1).
\end{align}
Furthermore,
for $s\in \{1,\ldots,a\}$ let $G_s$ be a graph
obtained from the disjoint union of $B_{b+(s-1)k+1},\dots,B_{b+sk}$ and
a $(k+1)$-cycle $v_{s,1}v_{s,2}\dots v_{s,k+1}v_{s,1}$ by identifying an edge of
$\tau(B_{b+(s-1)k+i})$ and $v_{s,i}v_{s,i+1}$ for each $i\in \{1,\cdots,k\}$.
Then $G_s$ is $C_k$-free and
\begin{align}\label{align.9}
|G_s|=\sum_{i=1}^{k}|B_{b+(s-1)k+i}|-(k-1)
\end{align}
for $s\in \{1,\dots,a\}$.

\begin{figure}[!ht]
	\centering
	\includegraphics[width=0.9\textwidth]{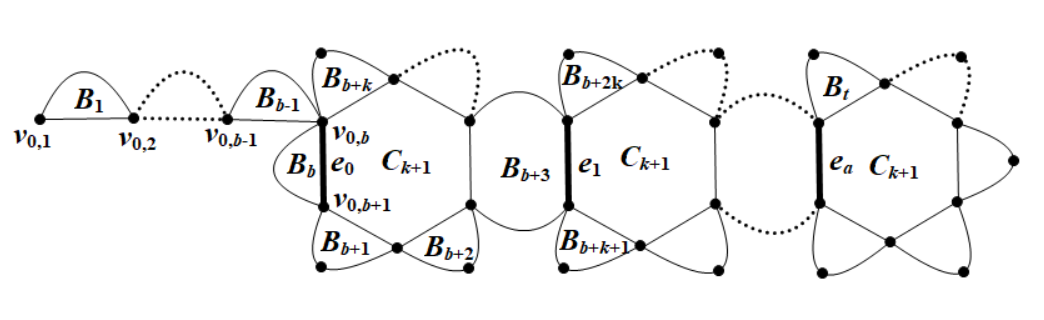}
	\caption{A construction of $G\in\mathscr{B}_k(B_1,\dots,B_t)$. }{\label{figu.5}}
\end{figure}

Let $e_0=v_{0,b}v_{0,b+1}$ and $e_s\in E(\tau(G_s))\setminus\{v_{s,k+1}v_{s,1}\}$ for $s\in \{1,\ldots,a\}$.
We now construct a graph $G$ from $G_0,G_1,\ldots,G_a$
by identifying $e_{s-1}$ and $v_{s,1}v_{s,k+1}$ for each $s\in \{1,\ldots,a\}$ (see Fig. \ref{figu.5}).
We can see that $G$ is a connected $C_k$-free outerplane graph with a decomposition of $k$-blocks $B_1,\dots,B_t$ and
\begin{align}\label{align.10}
|G|=\sum_{s=0}^a|G_s|-2a.
\end{align}
Combining with (\ref{align.8}-\ref{align.10}), we have
\begin{align*}
 |G|=\sum_{i=1}^{t}|B_{i}|-(b-1)-(k-1)a-2a=\sum_{i=1}^{t}|B_i|-\Big\lfloor\frac{t-1}{k}\Big\rfloor-t+1,
\end{align*}
since $a=\lfloor\frac{t-1}{k}\rfloor$ and $b=t-ka$.
\end{proof}

\begin{lem}\label{lemma2.3}
Let $G$ be a 2-connected outerplane graph with an inner face $F^*$.
If each inner face except $F^*$ is a $3$-face, then
$G$ contains $C_l$ for each $l\in \{|F^*|,\dots,|G|\}$.
\end{lem}

\begin{proof}
Let $\mathbb{C}$ be the set of cycles of $G$ that contain all vertices in $V(F^*)$.
Clearly, $\mathbb{C}\neq\varnothing$, since $F^*$ is a cycle.
We prove a stronger result, i.e., $\{C_{|F^*|},\dots,C_{|G|}\}\subseteq\mathbb{C}$.

Suppose to the contrary that $C_l\notin \mathbb{C}$
for some $l\in \{|F^*|,\dots,|G|\}$.
We may assume that $l$ is minimal.
Then $l\ge |F^*|+1$, and there is an $(l-1)$-cycle $C^*$ in $\mathbb{C}$.
Since $G$ is 2-connected, $\tau(G)$ is a Hamilton cycle.
Hence, there is an edge $uv\in E(C^*)\setminus  E(\tau(G))$ as $|C^*|=l-1<|G|$.
Moreover, $uv$ is incident with two inner faces of $G$,
as $uv\not\in  E(\tau(G))$.
Since all inner faces of $G$ except $F^*$ are 3-faces and $V(F^*)\subseteq V(C^*)$,
we can find an inner 3-face $F'$ such that
$ E(F')=\{uv,vw,wu\}$ and $w\not\in V(C^*)$.
Now define $H=G[V(C^*)\cup \{w\}]$.
Then $H\subseteq G$ and $H$ contains an $l$-cycle,
which contradicts the choice of $l$.
\end{proof}

We now definite two graph families.
Let $\mathcal{G}_{n,m,k}$ be the set of connected $C_k$-free outerplane graphs of order $n$ and size $m$
in which each inner face is a $k^-$-face;
and let $\mathcal{G}'_{n,m,k}$ be the set of connected $C_k$-free outerplane graphs of order $n$ and size $m$
in which each inner face is a $3$-face. Clearly, $\mathcal{G}'_{n,m,k}\subseteq \mathcal{G}_{n,m,k}$ for $k\geq4.$

\begin{lem}\label{lemma2.4}
Let $n,k$ be two integers with $n\ge2$ and $k\ge 4$.
If $\mathcal{G}_{n,m,k}$ is non-empty, then $\mathcal{G}'_{n,m,k}$ is also non-empty.
\end{lem}

\begin{proof}
We apply the induction on $n$.
If $n=2$, then $\mathcal{G}_{n,m,k}=\mathcal{G}'_{n,m,k}=\{K_2\}$, as desired.
Assume then that $n\ge 3$ and $G \in \mathcal{G}_{n,m,k}$.

If $G$ has cut-vertices,
then there exists an end-block $G_1$,
which contains exactly one cut-vertex $v$ of $G$.
Let $G_2:=G-V(G_1-v)$.
Clearly, $G_i\in \mathcal{G}_{|G_i|,e(G_i),k}$ for $i\in \{1,2\}$.
By the induction hypothesis, there exist two graphs $G_1',G_2'$ with $G_1'\in \mathcal{G}'_{|G_1|,e(G_1),k}$
and $G_2'\in \mathcal{G}'_{|G_2|,e(G_2),k}$.
Let $G^*$ be a graph obtained from $G_1'$ and $G_2'$ by identifying a vertex of $G_1'$ and a vertex of $G_2'$.
Then $|G^*|=|G_1|+|G_2|-1=n$ and $e(G^*)=e(G_1)+e(G_2)=m$.
It follows that $G^*\in \mathcal{G}'_{n,m,k}$.

Now assume that $G$ is 2-connected, and
let $\mathcal{F}_G$ be the set of inner faces of sizes at least 4 in $G$.
If $\mathcal{F}_G=\varnothing$,
then $G\in \mathcal{G}'_{n,m,k}$, as desired.
It remains $\mathcal{F}_G\not=\varnothing$.
Since $n\ge 3$ and $G$ is 2-connected,
by Fact \ref{fact2.1} $\tau(G)$ is a Hamilton cycle.
Let $F$ be an arbitrary inner face of size $s\ge 4$.
Then $ E(F)$ separates $G$ into $s$ outerplane subgraphs $G_1,\dots,G_s$ (see Fig. \ref{figu.3}).
For convenience to distinguish, let $G_{F}(e)$ denote the one containing $e$ for each $e\in E(F)$.

\begin{claim}\label{cl2.1}
There exist $F^*\in\mathcal{F}_G$ and $e^*\in  E(F^*)$
such that $\mathcal{F}_{G_{F^*}(e^*)}=\mathcal{F}_G \setminus\{F^*\}$.
\end{claim}

\begin{proof}
For each $F\in\mathcal{F}_G$ and each $e\in  E(F)$,
we can see that $F\not\in \mathcal{F}_{G_{F}(e)}$ and every inner face of $G_{F}(e)$ is also an inner face of $G$
(see Fig. \ref{figu.3}).
Hence, $\mathcal{F}_{G_{F}(e)}\subseteq\mathcal{F}_{G}\setminus \{F\}$.
Now let $F^*\in\mathcal{F}_G$ and $e^*\in  E(F^*)$ such that
$$|\mathcal{F}_{G_{F^*}(e^*)}|=\max\{|\mathcal{F}_{G_{F}(e)}|:~F\in\mathcal{F}_G ~~\mbox{and}~~e\in  E(F)\}.$$
Then $\mathcal{F}_{G_{F^*}(e^*)}\subseteq\mathcal{F}_{G}\setminus \{F^*\}$.
Next it suffices to show $\mathcal{F}_{G}\setminus \{F^*\}\subseteq\mathcal{F}_{G_{F^*}(e^*)}$.

\begin{figure}[!ht]
	\centering
	\includegraphics[width=0.4\textwidth]{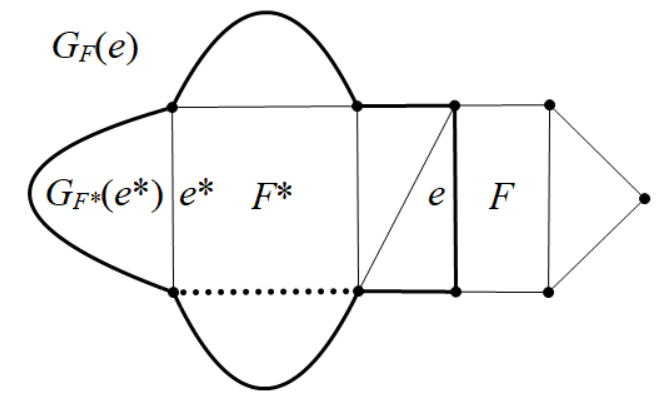}
	\caption{A 2-connected outerplane graph $G$ with $F\in\mathcal{F}_{G}\setminus (\{F^*\}\cup \mathcal{F}_{G_{F^*}(e^*)})$. }{\label{figu.6}}
\end{figure}

Suppose to the contrary that there exists an
$F\in\mathcal{F}_{G}\setminus (\{F^*\}\cup \mathcal{F}_{G_{F^*}(e^*)})$.
Then, there exists an edge $e\in  E(F)$ such that
$\mathcal{F}_{G_{F^*}(e^*)}\cup \{F^*\} \subseteq \mathcal{F}_{G_{F}(e)}$
(where $G_{F}(e)$ is bounded by bond lines, see Fig. \ref{figu.6}),
which contradicts the choices of $F^*$ and $e^*$.
\end{proof}

By Claim \ref{cl2.1},
there exist $F^*\in\mathcal{F}_G$ and $e^*\in  E(F^*)$
such that $\mathcal{F}_{G_{F^*}(e^*)}=\mathcal{F}_G \setminus\{F^*\}$.
Assume that $F^*:=v_1v_2\ldots v_sv_1$ and $e^*=v_1v_2$, where $s\ge 4$.
Let $G_i:=G_{F^*}(v_iv_{i+1})$ for $i\in \{1,\dots,s-1\}$ and $G_s:=G_{F^*}(v_sv_1)$ (see Fig. \ref{figu.4}).

\begin{claim}\label{cl2.2}
Let $G':=G[\cup_{i=2}^sV(G_i)]$. Then $|G'|\le k-1$.
\end{claim}

\begin{figure}[!ht]
	\centering
	\includegraphics[width=0.6\textwidth]{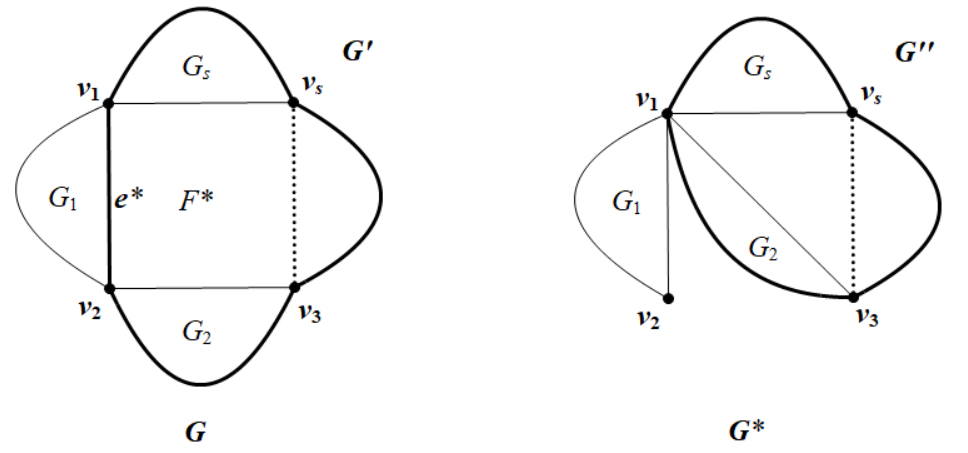}
	\caption{Examples of 2-connected outerplane graphs $G$ and $G^{*}$. }{\label{figu.4}}
\end{figure}

\begin{proof}
In Fig. \ref{figu.4}, $G'$ is bounded by bold lines.
Since $G\in\mathcal{G}_{n,m,k}$, each inner face of $G$ is a $k^-$-face,
and so $|F^*|=s\le k-1$.
By Claim \ref{cl2.1}, $\mathcal{F}_G=\mathcal{F}_{G_1}\cup\{F^*\}$.
Hence, $\mathcal{F}_{G_i}=\varnothing$ for each $i\in\{2,\dots,s\}$.
Thus, all inner faces of $G'$ except $F^*$ are 3-faces.
By Lemma \ref{lemma2.3}, $G'$ contains a cycle $C_l$ for each $l\in\{s,\dots,|G'|\}.$
Consequently, $|G'|\le k-1$, as $G'$ is $C_k$-free.
The claim holds.
\end{proof}


Let $G^{*}:=G-\{v_2v:v\in N_{G_2}(v_2)\}+\{v_1v:v\in N_{G_2}(v_2)\}$ (see Fig. \ref{figu.4}).
Then $G^{*}$ is still an outerplane graph with exactly two blocks $G_1$ and $G''$,
where $G''$ is bounded by bold lines. Clearly, $V(G'')=V(G')\setminus \{v_2\}$,
and by Claim \ref{cl2.2} $|G''|<|G'|\le k-1$.
Consequently, $G''$ is $C_k$-free and every inner face of $G''$ is a $k^-$-face.
Note that $G_1\subseteq G\in\mathcal{G}_{n,m,k}$.
Hence, $G_1$ is $C_k$-free and every inner face of $G_1$ is a $k^-$-face.
Therefore, $G^{*}\in\mathcal{G}_{n,m,k}$.
Now since $G^{*}$ has a cut-vertex $v_1$,
by previous discussion there exists a graph in $\mathcal{G}'_{n,m,k}$.
This completes the proof of Lemma \ref{lemma2.4}.
\end{proof}

We know that $e(G)=2|G|-3$ for every edge-maximal outerplane graph $G$ with $|G|\ge 2$.
It follows that ${\rm ex}_{\mathcal{OP}}(n,C_{k})=2n-3$ for $2\le n\le k-1$.
At the end of this section, we will determine ${\rm ex}_{\mathcal{OP}}(n,C_k)$
for $n\ge k\ge 3$.

\begin{thm}\label{thm2.2}
 Let $n,k,\lambda$ be integers with $n\geq k\ge 3$
 and $\lambda=\lfloor\frac{kn-2k-1}{k^2-2k-1}\rfloor+1$.
 Then ${\rm ex}_{\mathcal{OP}}(n,C_k)=f(n,k)$, where
 $$f(n,k)=\left\{
                                       \begin{array}{ll}
                                         2n-\lambda+2\big\lfloor\frac{\lambda}{k}\big\rfloor-3  & \hbox{$k\mid \lambda$,} \\
                                         2n-\lambda+2\big\lfloor\frac{\lambda}{k}\big\rfloor-2  & \hbox{otherwise.}
                                       \end{array}
                                     \right.
$$
\end{thm}

\begin{proof}
Let $G^*$ be a $C_k$-free outerplane graph of order $n$ with ${\rm ex}_{\mathcal{OP}}(n,C_k)$ edges.
Clearly, $G^*$ is connected.
Now we give some claims.

\begin{claim}\label{cl2.3}
$e(G^*)\geq f(n,k).$
\end{claim}

\begin{proof}
Let $\mathbb{Z}$ be the set of integers and
\begin{align}\label{align111}
S=\{s\in \mathbb{Z}: n+\Big\lfloor\frac{s-1}{k}\Big\rfloor+s-1\le s(k-1)\}.
\end{align}
Then for each $s\in  \mathbb{Z}$, $s\in S$ if and only if $\frac{s-1}{k}
<s(k-2)-n+2.$
It follows that
\begin{align}\label{align1111}
\min_{s\in S}s=\Big\lfloor\frac{kn-2k-1}{k^2-2k-1}\Big\rfloor+1=\lambda.
\end{align}
On the one hand, since $\lambda\in S$, we have
$n+\big\lfloor\frac{\lambda-1}{k}\big\rfloor+\lambda-1\le \lambda(k-1).$
Hence,
\begin{align}\label{align.11}
n+\Big\lfloor\frac{\lambda-1}{k}\Big\rfloor-(\lambda-1)(k-2)\le k-1.
\end{align}
On the other hand, since $\lambda-1\notin S$, we have
$n+\lfloor\frac{\lambda-2}{k}\rfloor+\lambda-2>(\lambda-1)(k-1)$.
Note that
$\lfloor\frac{\lambda-1}{k}\rfloor\geq\lfloor\frac{\lambda-2}{k}\rfloor.$
We can see that
\begin{align}\label{align.12}
n+\Big\lfloor\frac{\lambda-1}{k}\Big\rfloor-(\lambda-1)(k-2)>1.
\end{align}
Now we define a sequence of graphs as follows.
Let $B_i$
be the graph obtained by joining a new vertex with all vertices of a path of order $|B_i|-1$,
where $|B_i|=k-1$ for $i\in\{1,\dots,\lambda-1\}$
and $|B_{\lambda}|=n+\lfloor\frac{{\lambda}-1}{k}\rfloor-({\lambda}-1)(k-2)$.
Note that $k\geq3$. Combining with (\ref{align.11}-\ref{align.12}), we have
$2\leq|B_i|\leq k-1$ for each $i\in \{1,\ldots,\lambda\}$.
Consequently, every $B_i$ is either a copy of $K_2$
or a 2-connected $C_k$-free outerplane graph with each inner face a $3$-face.
This implies that every $B_i$ is a $C_k$-free trivial $k$-block.
By Lemma \ref{lemma2.2}, there exists a graph $G\in\mathscr{B}_k(B_1,\dots,B_\lambda)$
with minimal order, that is,
\begin{align}\label{align.13}
|G|=\sum_{i=1}^{\lambda}|B_i|-\Big\lfloor\frac{\lambda-1}{k}\Big\rfloor-\lambda+1.
\end{align}
It follows that $|G|=n$. Moreover,
by the definition of $B_i$,
we have $e(B_i)=2|B_i|-3$ for each $i\in\{1,\dots,\lambda\}$.
Combining with (\ref{align.13}), we have
\begin{align}\label{align.14}
e(G)=\sum_{i=1}^{\lambda}(2|B_i|-3)=2\Big(n+\Big\lfloor\frac{\lambda-1}{k}\Big\rfloor+\lambda-1\Big)-3\lambda.
\end{align}
If $k\nmid \lambda$, then $\lfloor\frac{\lambda-1}{k}\rfloor=\lfloor\frac{\lambda}{k}\rfloor,$
and so $e(G)=2n-\lambda+2\lfloor\frac{\lambda}{k}\rfloor-2.$
If $k\mid \lambda$, then we define $B_{\lambda+1}\cong K_2$.
By Lemma \ref{lemma2.2}, we can similarly find a graph $G'\in\mathscr{B}_k(B_1,\dots,B_{\lambda+1})$
with $|G'|=|G|=n$. However,
$e(G')=e(G)+1$.
Since now $\lfloor\frac{\lambda-1}{k}\rfloor=\lfloor\frac{\lambda}{k}\rfloor-1,$
by (\ref{align.14}) we have $e(G')=e(G)+1=2n-\lambda+2\lfloor\frac{\lambda}{k}\rfloor-3.$
The choices of $G$ and $G'$ imply that they are both $C_k$-free outerplane graphs of order $n$.
Therefore, above discussions give that $e(G^*)\geq f(n,k).$
\end{proof}

Having Claim \ref{cl2.3}, in the following it suffices to show $e(G^*)\leq f(n,k).$

\begin{claim}\label{cl2.4}
If $n_1<n_2$, then ${\rm ex}_{\mathcal{OP}}(n_1,C_k)<{\rm ex}_{\mathcal{OP}}(n_2,C_k)$.
\end{claim}

\begin{proof}
Let $G'$ be a $C_k$-free outerplane graph of order $n_1$ and size ${\rm ex}_{\mathcal{OP}}(n_1,C_k)$.
Let $G''$ be a graph obtained from $G'$ by joining a vertex of $G'$ to $n_2-n_1$ isolated vertices.
We can see that $G''$ is a $C_k$-free outerplane graph of order $n_2$.
Therefore, ${\rm ex}_{\mathcal{OP}}(n_1,C_k)<e(G'')\le {\rm ex}_{\mathcal{OP}}(n_2,C_k),$
as desired.
\end{proof}

Now let $\{B_1^*,\dots,B_t^*\}$ be the $k$-blocks decomposition of $G^*$.

\begin{claim}\label{cl2.5}
 $n=\sum_{i=1}^{t}|B_i^*|-\lfloor\frac{t-1}{k}\rfloor-t+1,$ where $n=|G^*|$.
\end{claim}

\begin{proof}
Set $n':=\sum_{i=1}^{t}|B_i^*|-\lfloor\frac{t-1}{k}\rfloor-t+1$.
By Lemma \ref{lemma2.1}, we have $n=|G^*|\geq n'$;
and by Lemma \ref{lemma2.2}, there exists a graph $G\in\mathscr{B}_k(B_1^*,\dots,B_t^*)$
with $|G|=n'$.
Since $G^*$ and $G$ have the same $k$-blocks decomposition,
$e(G)=\sum_{i=1}^{t}e(B_{i}^*)=e(G^*).$
Now if $n'<n$, then by Claim \ref{cl2.4},
$e(G)\leq{\rm ex}_{\mathcal{OP}}(n',C_k)<{\rm ex}_{\mathcal{OP}}(n,C_k)=e(G^*)$,
a contradiction.
Therefore, $n=n'$, as required.
\end{proof}

\begin{claim}\label{cl2.6}
There exists an extremal graph $G^{*}$
such that 
for each $k$-block its inner faces (if exist) are all 3-faces.
\end{claim}

\begin{proof}
Recall that every $k$-block is $K_2$ or a 2-connected subgraph of $G^*$ with each inner face a $k^-$-face.
Thus, $B_i^*\in \mathcal{G}_{|B_i^*|,e(B_i^*),k}$ for each $i\in \{1,\dots,t\}$.
If $k=3$, then every $k$-block $B_i^*\cong K_2$ and has no inner faces.
Then, $G^{*}$ itself is a desired graph.

It remains to consider $k\ge 4$.
By Lemma \ref{lemma2.4} for each $i\in \{1,\dots,t\}$ there exists a graph $G_i\in \mathcal{G}'_{|B_i^*|,e(B_i^*),k}$;
furthermore, $G_i$ admits a decomposition of $s_i$ $k$-blocks,
say $B_{i1},\ldots,B_{is_i}$, for some $s_i\ge 1$.
By the definition of $\mathcal{G}'_{|B_i^*|,e(B_i^*),k}$, every inner face of $G_i$ is a 3-face.
Thus, every inner face of $B_{ij}$ is also a 3-face by Fact \ref{fact2.2} (ii).
Moreover, by Lemma \ref{lemma2.1} we have
\begin{align}\label{align.16}
 |B_i^*|=|G_i|\ge \sum_{j=1}^{s_i}|B_{ij}|-\Big\lfloor\frac{s_i-1}{k}\Big\rfloor-s_i+1
\end{align}
for $i\in \{1,\dots,t\}$.
Combining with Claim \ref{cl2.5} and (\ref{align.16}), we have
\begin{align}\label{align.17}
 n\ge \!\Big(\sum_{i=1}^{t}\sum_{j=1}^{s_i}|B_{ij}|\!-
 \sum_{i=1}^{t}\Big\lfloor\frac{s_i-1}{k}\Big\rfloor
 \!-\!\sum_{i=1}^{t}s_i\!+\!t\Big)\!-\!\Big\lfloor\frac{t-1}{k}\Big\rfloor \!-\!t\!+\!1,
\end{align}
where
\begin{align}\label{align.18}
\sum_{i=1}^{t}\Big\lfloor\frac{s_i-1}{k}\Big\rfloor\le \Big\lfloor\sum_{i=1}^{t}\frac{s_i-1}{k}\Big\rfloor
\le \Big\lfloor\frac{\sum_{i=1}^{t}s_i-1}{k}\Big\rfloor-\Big\lfloor\frac{t-1}{k}\Big\rfloor.
\end{align}
Now let $G^{**}$ be an extremal graph with minimum number of vertices over all graphs in $\mathscr{B}_k(B_{11},\dots,B_{1s_1},\dots,B_{t1},\dots,B_{ts_t})$.
By Lemma \ref{lemma2.2}, we have
\begin{align}\label{align.19}
 |G^{**}|=\sum_{i=1}^{t}\sum_{j=1}^{s_i}|B_{ij}|-\Big\lfloor\frac{\sum_{i=1}^{t}s_i-1}{k}\Big\rfloor-\sum_{i=1}^{t}s_i+1.
\end{align}
Combining with (\ref{align.17}-\ref{align.19}), we have $|G^{**}|\le n$.
Moreover,
\begin{align}\label{align19}
e(G^{**})=\sum_{i=1}^{t}\sum_{j=1}^{s_i}e(B_{ij})=\sum_{i=1}^{t}e(G_i)=\sum_{i=1}^{t}e(B_i^*)=e(G^*).
\end{align}
If $|G^{**}|<n$, then we similarly have $e(G^{**})<e(G^*)$ by Claim \ref{cl2.4},
which contradicts (\ref{align19}).
Therefore, $|G^{**}|=n$, and so $G^{**}$ is a desired graph.
\end{proof}

Now we give the final proof of Theorem \ref{thm2.2}.
By Claim \ref{cl2.6}, we may assume that $G^*$ is an extremal graph with a decomposition of $k$-blocks
$B_1^*,\dots,B_t^*$ such that for each $B_i^*$ its inner faces (if exist) are all 3-faces.
We will see that $|B_{i}^*|\le k-1$ for each $i\in\{1,\ldots,t\}$.
If $B_i^*\cong K_2$, then we are done.
If $B_i^*\ncong K_2$, then by Lemma \ref{lemma2.3},
$B_{i}^*$ contains a cycle $C_l$ for each $l\in \{3,\dots,|B_{i}^*|\}$.
Consequently, $|B_{i}^*|\le k-1$, as $B_{i}^*$ is $C_k$-free.
Combining with Claim \ref{cl2.5}, we have
$n+\Big\lfloor\frac{t-1}{k}\Big\rfloor+t-1
=\sum_{i=1}^{t}|B_{i}^*|\le (k-1)t$.
By (\ref{align111}) and (\ref{align1111}), we have $t\in S$ and $t\ge \lambda$.

If $B_i^*\cong K_2$, then $e(B_{i}^*)=2|B_{i}^*|-3=1$.
If $B_i^*\ncong K_2$, we also have $e(B_{i}^*)=2|B_{i}^*|-3$,
as every inner face of $B_i^*$ is a 3-face.
It follows that $e(G^{*})=\sum_{i=1}^{t}e(B_{i}^*)=2\sum_{i=1}^{t}|B_{i}^*|-3t$.
Combining with Claim \ref{cl2.5}, we have
\begin{align*}
e(G^{*})=2n-t+2\Big\lfloor\frac{t-1}{k}\Big\rfloor-2.
\end{align*}

Let $f(t)=-t+2\lfloor\frac{t-1}{k}\rfloor$.
Then $f(t+1)-f(t)$ equals to $1$ if $k\mid t$, and -1 otherwise.
Since $t\ge \lambda$, we have
$f(t)\leq f(\lambda+1)$ for $k\mid\lambda$, and $f(t)\leq f(\lambda)$ otherwise.
Note that $\lfloor\frac{\lambda-1}{k}\rfloor=\lfloor\frac{\lambda}{k}\rfloor$ for $k\nmid\lambda$.
We can see that $e(G^*)\leq f(n,k)$.
Combining with Claim \ref{cl2.3}, we complete the proof of Theorem \ref{thm2.2}.
\end{proof}

\section{ Outerplanar Tur\'{a}n number of paths}

~~~~In this section we consider the outerplanar Tur\'{a}n number  ${\rm ex}_{\mathcal{OP}}(n,P_k)$
for any two positive integers $n$ and $k$ with $n\ge k\ge 3$.
Clearly, ${\rm ex}_{\mathcal{OP}}(n,P_3)=\lfloor\frac{n}{2}\rfloor$, and the extremal graph is the union of
$\lfloor\frac{n}{2}\rfloor$ isolated edges and at most one isolated vertex.
In the following we focus on $k\ge 4$.
We first introduce some notations.

Let $l(G)$ be the length of a longest path in a graph $G$.
Given $t$ integers $n_1,\dots,n_t$ with $n_1\ge \cdots \ge n_t\ge 2$,
a \emph{t-cactus} $C(n_1,\dots,n_t)$,
is obtained from $t$ maximal outerplanar graphs of orders $n_1,\dots,n_t$, respectively,
by sharing a vertex.
In particular, a 1-cactus is a maximal outerplanar graph.
By Fact \ref{fact2.1}, we have
\begin{align}\label{a20}
l(C(n_1,n_2,\dots,n_t))=\max_{1\leq i<j\leq t}(n_i+n_j-2)=n_1+n_2-2
\end{align}
for $t\ge 2$ and every $t$-cactus $C(n_1,\dots,n_t)$.
Denoted by ${\rm ex}'_{\mathcal{OP}}(n,P_k)$
the maximum number of edges in a connected $P_k$-free outerplanar graph on $n$ vertices.
Now we have the following lemma.

\begin{lem}\label{t31}
 Let $n$ and $k$ be two positive integers with $n\geq k\geq4$.
 Then
$${\rm ex}'_{\mathcal{OP}}(n,P_k)=2n-\Big\lceil\frac{n-k+1}{\lfloor\frac{k}{2}\rfloor-1}\Big\rceil-4. $$
\end{lem}

\begin{proof}
Let $G^*$ be a connected $P_{k}$-free outerplanar graph of order $n$ with maximal size.
Assume that $G^*$ has $t$ blocks $G_1,\dots,G_t$,
where $|G_i|=n_i$ for $i\in \{1,\ldots,t\}$ and
\begin{align}\label{a23}
n_1\ge \cdots \ge n_t\geq2.
\end{align}
If $t=1$, then $G^*$ itself is 2-connected.
By Fact \ref{fact2.1}, $G^*$ has a Hamilton cycle, and thus $G^*$ contains a copy of $P_k$,
which contradicts the choice of $G^*$.
Thus
\begin{align}\label{a25}
t\ge 2~~~\mbox{and}~~~\sum_{i=1}^t(n_i-1)=n-1.
\end{align}
Since every block of $C(n_1,\dots,n_t)$ is a maximal outerplanar graph, by (\ref{a25}) we have
\begin{align}\label{a27}
e\left(C(n_1,\dots,n_t)\right)=\sum_{i=1}^{t}(2n_i-3)=2\sum_{i=1}^{t}(n_i-1)-t=2n-t-2
\end{align}
for every $t$-cactus $C(n_1,\dots,n_t)$.
Furthermore, we have
\begin{align}\label{ali18}
e(G^*)=\sum_{i=1}^{t}e(G_i)\le \sum_{i=1}^{t}(2n_i-3)=e(C(n_1,\dots,n_t)).
\end{align}
Let $P$ be a path connecting $G_1$ and $G_2$ in $G$.
Then $l(G^*)\geq l(G_1)+l(P)+l(G_2)$.
By Fact \ref{fact2.1}, $l(G_i)=n_i-1$ for each $i\in \{1,\dots,t\}$.
Thus, $l(G^*)\geq  n_1+n_2-2$.
Combining with (\ref{a20}), we have $l(C(n_1,\dots,n_t))\leq l(G^*)$.
Since $G^*$ is $P_{k}$-free, $C(n_1,\dots,n_t)$ is also $P_{k}$-free.
By (\ref{a27}) and (\ref{ali18}), it suffices to determine the minimum $t$
such that we can find an $n$-vertex $P_{k}$-free $t$-cactus $C(n_1,\dots,n_t)$.
More precisely, we shall determine the minimum $t$ such that $n_1+n_2-2\le k-2$
and (\ref{a23}-\ref{a25}) hold.

It follows that
\begin{align}\label{a22}
n_1+n_2\le k ~~\mbox{and}~~n_i\le \Big\lfloor\frac{k}{2}\Big\rfloor
\end{align}
for $2\leq i\leq t.$
Furthermore, combining with (\ref{a25}) and (\ref{a22}) we have
$$n+t-1=\sum_{i=1}^tn_i=(n_1+n_2)+\sum_{i=3}^{t}n_i\le k+(t-2)\Big\lfloor\frac{k}{2}\Big\rfloor.$$
This implies that
\begin{align}\label{al2}
\min t=\Big\lceil\frac{n-k+2\lfloor\frac{k}{2}\rfloor-1}{\lfloor\frac{k}{2}\rfloor-1}\Big\rceil=
\frac{n-k+2\lfloor\frac{k}{2}\rfloor-1}{\lfloor\frac{k}{2}\rfloor-1}+\alpha,
\end{align}
where $\alpha\in [0,1)$ and $\alpha(\lfloor\frac{k}{2}\rfloor-1)$ is an integer.
Therefore, $0\leq\alpha(\lfloor\frac{k}{2}\rfloor-1)\leq \lfloor\frac{k}{2}\rfloor-2$ and hence
\begin{align}\label{al3}
2\leq\Big\lfloor\frac{k}{2}\Big\rfloor-\alpha\Big(\Big\lfloor\frac{k}{2}\Big\rfloor-1\Big)
\leq\Big\lfloor\frac{k}{2}\Big\rfloor.
\end{align}
Now set $n_1=\Big\lceil\frac{k}{2}\Big\rceil$, $n_i=\lfloor\frac{k}{2}\rfloor$ for $2\leq i\leq \min t-1$
and $n_{\min t}=\lfloor\frac{k}{2}\rfloor-\alpha(\lfloor\frac{k}{2}\rfloor-1).$
By (\ref{al3}), the choices of $n_1,n_2,\ldots,n_{\min t}$ satisfy (\ref{a23}) and (\ref{a22}).
Furthermore, we will see that the choices also satisfy (\ref{a25}).

On the one hand, by (\ref{al2}) we have
$$
\min t\ge
\frac{n-k+2\lfloor\frac{k}{2}\rfloor-1}{\lfloor\frac{k}{2}\rfloor-1}
=\frac{n-k+1}{\lfloor\frac{k}{2}\rfloor-1}+2>2,
$$
as $n\ge k$ and $k\ge 4$.
On the other hand,
\begin{eqnarray*}\label{a26}
\sum_{i=1}^{\min t}(n_i-1)
  &=& (n_1+n_2)+\sum_{i=3}^{\min t-1}n_i+n_{\min t}-\min t\nonumber\\
  &=& k+(\min t-3)\Big\lfloor\frac{k}{2}\Big\rfloor+\Big(\Big\lfloor\frac{k}{2}\Big\rfloor-
  \alpha\Big(\Big\lfloor\frac{k}{2}\Big\rfloor-1\Big)\Big)-\min t \nonumber\\
  &=&  \Big(k-2\Big\lfloor\frac{k}{2}\Big\rfloor\Big)+(\min t-\alpha)\Big(\Big\lfloor\frac{k}{2}\Big\rfloor-1\Big) \nonumber\\
  &=&  n-1,
\end{eqnarray*}
as by (\ref{al2}) we have $(\min t-\alpha)\left(\lfloor\frac{k}{2}\rfloor-1\right)=n-k+2\lfloor\frac{k}{2}\rfloor-1$.

By above discussions, we find an $n$-vertex $P_k$-free $(\min t)$-cactus $C(n_1,\dots,n_{\min t})$,
where $n_1=\lceil\frac{k}{2}\rceil$, $n_i=\lfloor\frac{k}{2}\rfloor$ for $2\leq i\leq \min t-1$
and $n_{\min t}=\lfloor\frac{k}{2}\rfloor-\alpha(\lfloor\frac{k}{2}\rfloor-1).$
It follows that ${\rm ex}'_{\mathcal{OP}}(n,P_{k})=
e(C(n_1,\dots,n_{\min t}))$.
Combining with (\ref{a27}) and (\ref{al2}),
$${\rm ex}'_{\mathcal{OP}}(n,P_{k})=
e(C(n_1,\dots,n_{\min t}))
=2n-{\min t}-2
=2n-\Big\lceil\frac{n-k+1}{\lfloor\frac{k}{2}\rfloor-1}\Big\rceil-4,$$
as desired.
\end{proof}

Now let ${\rm ex}''_{\mathcal{OP}}(n,P_k)$ be the maximum number of edges in a $P_k$-free outerplanar graph on $n$ vertices in which every component is of order at most $k-1$.

\begin{lem}\label{t32}
Let $n$ and $k$ be two positive integers with $n\ge k\ge 4$. Then
 $${\rm ex}''_{\mathcal{OP}}(n,P_k)=\left\{
                                       \begin{array}{ll}
                                         2n-3\lceil\frac{n}{k-1}\rceil+1  & \hbox{if $n\equiv1\pmod{k-1}$,} \\
                                          2n-3\lceil\frac{n}{k-1}\rceil  & \hbox{otherwise.}
                                       \end{array}
                                     \right.
$$
\end{lem}

\begin{proof}
Let $G^*$ be a $P_{k}$-free outerplanar graph of order $n$ such that:
({\romannumeral1}) every component is of order at most $k-1$;
({\romannumeral2}) subject to (i), $e(G^*)$ is maximal;
({\romannumeral3}) subject to (i) and (ii), $G^*$ has maximal number of
components  with $k-1$ vertices.

Assume that $G^*$ has $s$ components $G_1,\ldots,G_s$, where $|G_i|=n_i\leq k-1$ for $i\in \{1,\ldots,s\}$
and $n_1\ge \dots \ge n_s$.
By the definition of $G^*$, $n_i+n_j\geq k$ for any two components $G_i$ and $G_j$.
Moreover, recall that $e(G)=2|G|-3$ for every maximal outerplanar graph $G$ with $|G|\geq2$.
Thus, $e(G_{i})=2n_i-3$ for every non-trivial component $G_i$.

In the following we first show that $n_i=k-1$ for each $i\le s-1$.
Suppose to the contrary that $n_{i_0}\le k-2$ for some $i_0\leq s-1$.
Then $n_j\geq 2$ for every $j\neq i_0$.
Consequently, $n_{i_0}\geq n_s\geq2$
and $e(G_{i_0})+e(G_s)=2(n_{i_0}+n_s)-6$.
Now, define two new components $G'_{i_0}$ and $G'_s$ as follows:
$G'_{i_0}$ and $G'_s$ are maximal outerplanar graphs with $|G'_{i_0}|=k-1$ and $|G'_s|=(n_{i_0}+n_s)-(k-1)$. Note that $k\leq n_{i_0}+n_s\leq 2(k-2)$. Then $1\leq |G'_s|\leq k-3.$ Moreover,
$e(G'_{i_0})=2|G'_{i_0}|-3$ and $e(G'_s)\in \{2|G'_s|-3, 2|G'_s|-2\}$, since $|G'_{i_0}|=k-1\geq3$ and $|G'_s|\geq1$.
Let $G'$ be the graph obtained from $G^*$
by replacing components $G_{i_0}$ with $G'_{i_0}$ and $G_s$ with $G'_s$.
Thus, $$e(G')-e(G^*)=(e(G'_{i_0})+e(G'_s))-(e(G_{i_0})+e(G_s))\geq0.$$
However, $G'$ has a new component $G'_{i_0}$ of order $k-1$, which contradicts the choice of $G^*$.

Now we have that $n_i=k-1\ge 3$ for $i\le s-1$.
Hence, $s=\lceil\frac{n}{k-1}\rceil$ and
 $${\rm ex}''_{\mathcal{OP}}(n,P_k)=\left\{
                                       \begin{array}{ll}
                                        2n-3s+1  & \hbox{if $n_s=1$,} \\
                                         2n-3s  & \hbox{if $n_s\geq2$.}
                                       \end{array}
                                     \right.
$$
The desired result follows.
\end{proof}

Having Lemmas \ref{t31} and \ref{t32} in hand, we are ready to give the outerplanar Tur\'{a}n numbers of paths.

\begin{thm}\label{t34}
Let $n$ and $k$ be two positive integers with $n\ge k\ge 4$. Then
$${\rm ex}_{\mathcal{OP}}(n,P_k)=
\max\{{\rm ex}'_{\mathcal{OP}}(n,P_k),{\rm ex}''_{\mathcal{OP}}(n,P_k)\}.$$
\end{thm}

\begin{proof}
Let $G^*$ be a $P_{k}$-free outerplanar graph of order $n$ with maximal size.
Assume that $G^*$ has $s$ components $G_1,\ldots,G_s$, where $|G_i|=n_i$ for $i\in \{1,\ldots,s\}$
and $n_1\ge \dots \ge n_s$.
We consider two cases.

\noindent{{\bf{Case 1.}}} $k\in \{4,5\}$.

We have $\lfloor\frac{k}{2}\rfloor=2$, and we may assume that
$G^*$ is an extremal graph such that:
({\romannumeral1}) $e(G^*)$ is maximal;
({\romannumeral2}) subject to the condition (i), $G^*$ has maximal number of components with $k-1$ vertices.

We first claim that $n_1\leq k-1$.
Suppose to the contrary that $n_1\ge k$.
Then by Lemma \ref{t31}, we have $$e(G_1)={\rm ex}'_{\mathcal{OP}}(n_1,P_k)=n_1+k-5.$$
Now, we define $G'_1$ and $G''_1$ as follows:
$G'_1$ is a maximal outerplanar graph with $k-1$ vertices,
and $G''_1$ is a connected $P_k$-free outerplanar graph on $n_1-k+1$ vertices with maximal number of edges.

Obviously, $e(G'_1)=2|G'_1|-3=2k-5.$ Moreover,
if $n_1=k$, then $e(G''_1)=0;$
if $k+1\leq n_1\le 2k-2$, then $G''_1$ is also a non-trivial maximal outerplanar graph
and hence $e(G''_1)=2|G'_1|-3=2n_1-2k-1;$
if $n_1\ge 2k-1$,
then $e(G''_1)={\rm ex}'_{\mathcal{OP}}(n_1-k+1,P_k)=n_1-4$ by Lemma \ref{t31}.
It follows that $e(G''_1)\geq n_1-k$ and $e(G'_1)+e(G''_1)\geq n_1+k-5$.

Let $G'$ be the graph obtained from $G^*$
by replacing component $G_1$ with two new components $G'_1$ and $G''_1$.
Then $e(G'_1)+e(G''_1)\geq n_1+k-5\geq e(G_1).$
It follows that $e(G')\geq e(G^*).$
However, $G'$ has a new component $G'_1$ of order $k-1$, which contradicts the definition of $G^*$.
Therefore, the previous claim holds, that is, $n_i\leq k-1$ for each $i$.
By Lemma \ref{t32}, we have $e(G^*)={\rm ex}''_{\mathcal{OP}}(n,P_k)$.

\noindent{{\bf{Case 2.}}} $k\ge 6$.

We may assume that $G^*$ is an extremal graph such that:
({\romannumeral1}) $e(G^*)$ is maximal;
({\romannumeral2}) subject to the condition (i), $G^*$ has minimal number of components.

We first assume that $s\ge 2$ and $n_1\ge k$.
By Lemma \ref{t31}, we have $$e(G_1)={\rm ex}'_{\mathcal{OP}}(n_1,P_k)=2n_1-\Big\lceil\frac{n_1-k+1}{\lfloor\frac{k}{2}\rfloor-1}\Big\rceil-4.$$
Let $G'_{1}$ be a connected $P_k$-free outerplanar graph on $n_1+n_s$ vertices with maximal number of edges.
Similarly, $$e(G'_{1})={\rm ex}'_{\mathcal{OP}}(n_1+n_s,P_k)=2(n_1+n_s)-
\Big\lceil\frac{(n_1+n_s)-k+1}{\lfloor\frac{k}{2}\rfloor-1}\Big\rceil-4.$$
Note that
 $$e(G_s)=\left\{
                                       \begin{array}{ll}
                                        0  & \hbox{if $n_s=1$,} \\
                                        2n_s-3  & \hbox{if $2\le n_s\le k-1$,}\\
                                        2n_s-\big\lceil\frac{n_s-k+1}{\lfloor\frac{k}{2}\rfloor-1}\big\rceil-4  & \hbox{if $n_s\geq k$.}
                                       \end{array}
                                     \right.
$$
By straightforward computation, we can find that $e(G_1)+e(G_s)\le e(G'_1).$
Now let $G'$ be the graph obtained from $G^*$
by replacing two components $G_1$ and $G_s$ with a new component $G'_1$.
Then $e(G')\geq e(G^*).$
However, $G'$ has exactly $s-1$ components, which contradicts the choice of $G^*$.
Therefore, $s=1$ or $n_1\le k-1$.

If $s=1$, that is, $G^*$ is connected,
then by Lemma \ref{t31} we have $e(G^*)={\rm ex}'_{\mathcal{OP}}(n,P_k)$.
If $n_1\le k-1$, then by Lemma \ref{t32} we have $e(G^*)={\rm ex}''_{\mathcal{OP}}(n,P_k)$.
This completes the proof.
\end{proof}

We end this section with the following remark.

\noindent{{\bf{Remark.}}} We can observe in the proof of Theorem \ref{t34} that ${\rm ex}'_{\mathcal{OP}}(n,P_k)\le {\rm ex}''_{\mathcal{OP}}(n,P_k)$ for $k\in \{4,5\}$.
For $k=7$, direct computation gives that
 $${\rm ex}'_{\mathcal{OP}}(n,P_k)\left\{
                                       \begin{array}{ll}
                                        >{\rm ex}''_{\mathcal{OP}}(n,P_k)  & \hbox{if $n\equiv2\pmod{6}$,} \\
                                        <{\rm ex}''_{\mathcal{OP}}(n,P_k)  & \hbox{if $n\equiv0,5\pmod{6}$,}\\
                                        ={\rm ex}''_{\mathcal{OP}}(n,P_k)  & \hbox{if $n\equiv1,3,4\pmod{6}$.}
                                       \end{array}
                                     \right.
$$
For $k=6$ or $k\ge 8$ and sufficiently large $n$, we can see that
${\rm ex}'_{\mathcal{OP}}(n,P_k)> {\rm ex}''_{\mathcal{OP}}(n,P_k)$.


\end{document}